\documentclass[11pt,oneside,reqno]{amsart}
\usepackage{mathpazo} % math & rm
\normalfont
\usepackage[T1]{fontenc}

\usepackage[utf8]{inputenc}
\usepackage{amsmath}
\usepackage{amsfonts}
\usepackage{amssymb}
\usepackage{amsthm}
\usepackage{mathtools}%gives the arrow typeset for inclusion
\usepackage{stmaryrd}%Including this for the "double-bracket" \llbracket and \rrbracket
\usepackage{tikz}
\usepackage{tikz-cd}
\usepackage[all]{xy}
\usepackage{enumitem}%This package allows me to change the formatting of the numbers for an enumerate environment.
\usepackage[margin=1.2in]{geometry} % to set the margins

    \newtheorem{thm}{Theorem}[section]
    \newtheorem{theorem}[thm]{Theorem}
    \newtheorem{lemma}[thm]{Lemma}

    \newtheorem{corollary}[thm]{Corollary}
    
    \newtheorem{proposition}[thm]{Proposition}
    
    \theoremstyle{definition}
    
    \newtheorem{definition}[thm]{Definition}
    \newtheorem{remark}[thm]{Remark}
    \newtheorem{example}[thm]{Example}
    \newtheorem{notation}[thm]{Notation}

    \newtheorem{innercustomthm}{Theorem}
    \newenvironment{customthm}[1]
    {\renewcommand\theinnercustomthm{#1}\innercustomthm}
    {\endinnercustomthm}

    \newtheorem{nono-theorem}{Theorem}[]
    \newtheorem{nono-coro}{Corollary}[]

    \newcommand{\polyFromBend}[1]{\mathcal{I}_{#1}}

    \def\closed{\rm{closed}}
    \def\iclosed{\it{closed}}

    \def\cocoa{{\hbox{\rm C\kern-.13em o\kern-.07em C\kern-.13em o\kern-.15em A}}}
    
    \def\mz{{\mathbb{Z}}}
    \def\mn{{\mathbb{N}}}
    
    \def\mr{{\mathbb{R}}}

    \def\RR{{\mathbb{R}}}
         
    \def\mb{{\mathbb{B}}}

    \def\b0{{\pmb 0}}

    \def\vecx{{\pmb x}}
    \def\vecu{{\pmb u}}

    \def\rma{{\mathbb{T}}}
    \def\TT{{\mathbb{T}}}

    \def\VV{\mathbb{V}}
    \def\VC{\mathbb{V}(C)}
    
    \def\II{\mathbb{I}}
    \def\tPoly{{\mathbb{T}[x_1, \dots, x_n]}}

    \def\tLaur{{\mathbb{T}[x_1^{\pm 1}, \dots, x_n^{\pm 1}]}}

    \def\<{\langle}
    \def\>{\rangle}
\begin{document}

    %\title{Varieties of closed prime ideals}

    \title{Varieties of prime tropical ideals and the dimension of the coordinate semiring}

    \author{D\'aniel Jo\'o}
    \address{Mollia, Budapest, Hungary}
    \email{jood.mollia@gmail.com}

    \author{Kalina Mincheva}

\address{Department of Mathematics, Tulane University, New Orleans, LA 70115}
\email{kmincheva@tulane.edu}

%    General info


\begin{abstract}
In this note we study the relationship between ideals and congruences of the tropical polynomial and Laurent polynomial semirings. We show that the variety of a non-zero prime ideal of the tropical (Laurent) polynomial semiring consists of at most one point. We also prove a result relating the dimension of an affine tropical variety and the dimension of its ``coordinate semiring''.
\end{abstract}

\maketitle

\section{Introduction}
Tropical geometry studies algebraic varieties through their combinatorial shadows -- tropical varieties (defined by tropical polynomials) -- thus allowing for the use of a broader range of (combinatorial) tools. Some applications of tropical methods include computing Gromov-Witten invariants due to Mikhalkin in \cite{Mik06}, tropical proof of the Brill-Noether Theorem \cite{BrillNoether}, Brill-Noether theory for curves of a fixed gonality \cite{JR17}, developing a strategy to attack the Riemann hypothesis \cite{CC18}, the Gross-Siebert program in mirror symmetry \cite{Gro11}, and studying toric degenerations \cite{KM16}.\par

In tropical geometry most algebraic computations are done on the classical side by using the algebra of the original variety. The theory developed so far has explored the geometric aspect of tropical varieties as opposed to the underlying (semiring) algebra. There are still many commutative algebra tools and notions without a tropical analogue. Understanding the underlying algebraic structures can complement the existing tropical methods. For instance, the algebraic data associated to a tropical variety is often the only way to determine if a geometric object with the combinatorial properties of a tropical variety arises from an algebraic variety. The fact that this is not always the case can be a major obstruction to using tropical methods to solve enumerative problems. 

In the recent years, there has been a lot of effort dedicated to developing the necessary tools for commutative algebra using different frameworks. Among these are the arithmetic and scaling cite \cite{CC14}, \cite{CC16}, blueprints \cite{Lor15},  tropical ideals \cite{MR14}, \cite{MR16}, tropical schemes \cite{GG13}, super-tropical algebra \cite{IR10}, systems \cite{R18}, and prime congruences \cite{JM14}. These approaches aim for the exploration of the properties of tropicalized spaces without tying them up to the original varieties and working with geometric structures inherently defined in characteristic one, that is, over additively idempotent semifields. We are mostly interested in the case when the underlying semifield is the tropical semifield $\rma$. It is the set $\{ \mathbb{R} \cup \{-\infty\}\}$ with two operations: maximum playing the role of addition and addition acting as multiplication. 

In the polynomial semiring and the Laurent polynomial semiring over $\TT$, there is no bijection between ideals and congruences, equivalence relations that respect the operations. For this reason we often need to work with both ideals and congruences. To an ideal, in fact to any set, one can associate a congruence relation, having geometric meaning, called a bend congruence. These congruences were originally introduced in \cite{GG13}, see Definition~\ref{def:bend_rel}. Bend congruences are generated by the so called bend relations of all polynomials in the ideal (or set). Conversely, to a congruence $C$, one can associate the ideal $\polyFromBend{C}$ of all polynomials, whose bend relations are contained in the congruence. 
If $\polyFromBend{Bend(I)} = I$ then we say that the ideal $I$ is $\iclosed$. 
%The ideals that satisfy the property $\polyFromBend{Bend(I)} = I$ we call \emph{closed}. 
Examples of closed ideals are the \emph{tropical ideals} as shown in \cite{MR14}. Tropical ideals are arguably the most important class of ideals as they fully capture the algebra of tropical varieties. However, tropical ideals often cannot be described by a finite amount of data which makes working with them hard.

%Tropical ideals are ideals of the polynomial semiring with coefficients in $\rma$, which satisfy certain combinatorial conditions. More precisely, the polynomials of degree at most $d$ in a tropical ideal can be viewed as the vectors of a valuated matroid. An equivalent description is given in Definition~\ref{def: trop_ideals}. Tropical ideals fully capture the algebra of tropical varieties, which are realized as the vanishing locus of an ideal of tropical polynomials. We refer to ideals as realizable tropical ideals whenever they arise as the image of classical ideals. However, not every tropical ideal comes as the tropicalization of some classical object.

%In this note we give some results on prime ideals in the tropical polynomial and Laurent polynomial semirings and their varieties.
In our first two results we classify prime tropical ideals and the varieties of all prime ideals. We use the relation between ideals and the associated bend congruences, and the theory of prime and geometric prime congruences developed in our previous work \cite{JM14}. A \textit{prime} congruence $C$ on $\tPoly$ (resp. $\tLaur$) is an equivalence relation that respects the operations and the quotient $\tPoly/C$ (resp. $\tLaur/C$) is a totally ordered cancellative semiring. A prime congruence is a \textit{geometric prime} if this quotient is isomorphic to $\TT$.
% \begin{customthm}{A}
% The ideal $\polyFromBend{P} \subseteq \tLaur$ is prime if and only if $P$ is a prime congruence of $\tLaur$.
% \end{customthm}

\begin{customthm}{A}
Every closed prime ideal in $\tLaur$ is of the form $\polyFromBend{P}$ for a prime congruence $P$. $\polyFromBend{P}$ is the zero ideal if and only if $P$ is a minimal prime. The maps $I \rightarrow Bend(I)$ and $P \rightarrow \polyFromBend{P}$ establish a one to one correspondence between the non-zero closed prime ideals and the non-minimal prime congruences of $\tLaur$.
The closed prime ideal $\polyFromBend{P}$ is a tropical ideal if and only if $P$ is a geometric prime congruence. 
\end{customthm}

If $I$ is an ideal of $\tLaur$ (resp. $\tPoly$) we define the variety $\VV(I)$ of $I$ as the set of all points in $\RR^n$ (resp. $\TT^n$) where every polynomial $f\in I$ tropically vanishes, i.e., its maximum as a function is attained at two or more different terms.

\begin{customthm}{B}
If $I$ is a non-zero prime ideal of $\tLaur$ or $\tPoly$ then $\VV(I)$ is at most a point. In particular, the varieties of prime tropical ideals are just points as only tropical prime ideals are the tropicalizations of maximal ideals. 
\end{customthm}

We also relate the dimension of an affine tropical variety (of a tropical ideal) and the dimension of its ``coordinate ring''. The results can be summarized as follows:

\begin{customthm}{C}\footnote{This result for realizable tropical ideals was originally announced in the second author's PhD thesis.}
Let $I$ be a tropical ideal in $\tPoly$ such that $\VV(I)$ has dimension $d$ as a polyhedral complex. Then 
$$\dim \tPoly/Bend(I) = d + 1,$$
where $\dim \tPoly/Bend(I)$ is the number of strict inclusions in the longest chain of prime congruences. 
\end{customthm}

Intuitively this result states that the (Krull) dimension of the tropical ``coordinate ring" is the same as the geometric dimension, corrected for the dimension of $\mathbb{T}$, which is 1. Theorem B can be seen as an algebraic way to compute the dimension of the tropicalization of an affine variety.

%\orange{Theorem B can be seen as an algebraic way to compute the dimension of the tropicalization of an affine variety. Let $K$ be a field with valuation and let $I$ an ideal of $K[x_1, \dots, x_n]$ and let $trop(I) = \{ trop(f) \ : \ f\in I \}$, where $trop(f)$ is the polynomial in $\tPoly$ obtained from $f$ by valuating all coefficients of $f$ and replacing the operations multiplication and addition with their tropical counterparts. Then $$\dim trop( V(I)) = \dim \tPoly/Bend(trop(I))-1.$$ As before, $\dim \tPoly/Bend(\nu(I))$ is the number of strict inclusions in the longest chain of prime congruences.}

\addtocontents{toc}{\protect\setcounter{tocdepth}{-1}}
\section*{Acknowledgements}

The authors thank Diane Maclagan and Nati Friedenberg for their insightful comments on earlier versions of this paper. K.M. acknowledges the support of the Simons Foundation, Travel Support for Mathematicians.

\addtocontents{toc}{\protect\setcounter{tocdepth}{2}}

\section{Preliminaries}

In this paper by a {\it semiring} we mean a commutative semiring with multiplicative unit, that is a nonempty set $R$ with two binary operations $(+,\cdot)$ satisfying:

\begin{itemize}
\item $(R,+)$ is a commutative monoid with identity element $0_R$
\item $(R,\cdot)$ is a commutative monoid with identity element $1_R$
\item For any $a,b,c \in R$: $a(b+c) = ab+ac$
\item $1_R \neq 0_R$ and $a\cdot 0_R = 0_R$ for all $a \in R$
\end{itemize}

A {\it semifield} is a semiring in which all nonzero elements have multiplicative inverse. We will denote by $\mb$ the semifield with two elements $\{1,0\}$, where $1$ is the multiplicative identity, $0$ is the additive identity and $1+1 = 1$. The {\it tropical semifield}  $\rma$ is defined on the set $\mr  \cup \{-\infty\} $, by setting the $+$ operation to be the usual maximum and the $\cdot$ operation to be the usual addition, with $-\infty = 0_\rma$.
By a {\it $\mb$-algebra} $A$ we simply mean a commutative semiring with idempotent addition (that is $a + a = a, \forall a \in A$). Note that the idempotent addition defines an ordering via $a \geq b \iff a+b = b$. We will call a $\mb$-algebra {\it totally ordered} if its addition induces a total ordering.
%In this note we denote $\mb$-algebras by $A$ and semirings by $R$.
%% Note to self: we denote with R semirings and with A B-algebras, just to emphasize 

\par\smallskip

A \textit{polynomial (resp. Laurent polynomial)} semiring with variables $x_1,\dots,x_n$ over a semifield $F$ is denoted by $F[x_1,\dots,x_n]$ (resp. $F[x_1^{\pm 1},\dots,x_n^{\pm 1}]$). Its elements are formal linear combinations of the monomials $\{x_1^{k_1}...x_k^{k_n}\mid\;k_i \in \mn\}$ (resp. $\{x_1^{k_1}...x_n^{k_n}\mid\;k_i \in \mz\}$) with coefficients in $F$, with addition and multiplication being defined in the usual way. By a \textit{term} in $F[x_1,\dots,x_n]$ (resp. $F[x_1^{\pm 1},\dots,x_n^{\pm 1}]$) we will mean a monomial with a coefficient.\par\smallskip
As usual, an {\it ideal} of a semiring $R$ in additive submonoid that is closed under multiplication by any element of $R$.
%the semiring $R$ is just a subsemiring that is closed under multiplication by any element of $R$. Congruences of semirings are just operation preserving equivalence relations.

\begin{definition}
A {\it congruence} $C$ of the semiring $R$ is a subset of $R \times R$, such that it is an equivalence relation that respects the operations on $R$. In other words:
\begin{itemize}
\item For $a \in R$, $(a,a) \in C$
\item $(a,b) \in C$ if and only if $(b,a) \in C$
\item If $(a,b) \in C$ and $(b,c) \in C$ then $(a,c) \in C$
\item If $(a,b) \in C$ and $(c,d) \in C$ then $(a+c,b+d) \in C$
\item If $(a,b) \in C$ and $(c,d) \in C$ then $(ac,bd) \in C$
\end{itemize}
\end{definition}

The unique smallest congruence is the diagonal of $R \times R$, also called the {\it trivial congruence}. % and denoted by $\diag$. 
It often plays the role of the zero ideal. $R \times R$ itself is the {\it improper congruence} the rest of the congruences are called {\it proper}. Elements of $R \times R$ are called {\it pairs}. 

\begin{definition}[adapted from Definition 5.1.1 in \cite{GG13}]\label{def:bend_rel}{\rm
Let $A$ be a $\mb$-algebra and $f \in R[x_1, \ldots, x_n]$. For $i$ in the support of $f$ denoted by $supp(f)$, we write $f_{\hat \imath}$ for the result of deleting the $i$ term from $f$. Then the {\it bend relations of $f$} is the set of pairs 
$$bend(f) = \{(f , f_{\hat \imath})\}_{i\in supp(f)}.$$

The {\it bend congruence} is the congruence generated by the bend relations of a set $I$ of polynomials in $A[x_1, \ldots, x_n]$, denoted $Bend(I)$. We will be mostly interested in the case when $I$ is an ideal.
}
\end{definition}

%We denote pairs by Greek letters, and denote the coordinates of the pair $\alpha$ by $\alpha_1,\alpha_2$. 
\noindent Let $R$ be a semiring. The {\it twisted product} of the pairs $\alpha = (\alpha_1, \alpha_2)$ and $\beta = (\beta_1, \beta_2)$ of $R \times R$ is $$\alpha\beta=(\alpha_1\beta_1+\alpha_2\beta_2,\alpha_1\beta_2+\alpha_2\beta_1).$$ Note that the twisted product is associative. The set of pairs is a monoid under this operation, with the pair $(1,0)$ being the identity element. For the rest of the paper in any formula containing pairs the product is always the twisted product, so the twisted product of $\alpha$ and $\beta$ is simply denoted by $\alpha\beta$. Similarly $\alpha^n$ denotes the twisted $n$-th power of the pair $\alpha$, and we use the convention $\alpha^0 = (1,0)$. %The product of two congruences $C$ and $E$ is defined as the congruence generated by the set $\{\alpha\beta\mid \alpha \in C,\  \beta\in E\}$. 
For an element $a$ and a pair $\alpha$ we define their product as $a(\alpha_1,\alpha_2)=(a\alpha_1,a\alpha_2)$ which is the same as the twisted product $(a,0)\alpha$.\\

We recall some definitions and results from \cite{JM14}.

\begin{proposition}[Proposition 2.2 in \cite{JM14}]\label{prop: closedTwProd}
Let $C$ be a congruence of a $\mb$-algebra $A$, for $\alpha \in C$ and an arbitrary pair $\beta$ we have $\alpha\beta \in C$.
\end{proposition}

\begin{definition}\label{def:qc-p-r-prime}
A congruence $C$ on a $\mb$-algebra $R$ is called
\begin{itemize}
    \item {\it quotient cancellative} if $R/C$ is a cancellative semiring, namely if $ac = bc \in R/C$ then either $a=b$ or $c = 0_{R/C}$.
    \item {\it prime} if it is proper and for every $\alpha, \beta \in R \times R$ such that $\alpha\beta \in C$ either $\alpha \in C$ or $\beta \in C$. %Equivalently, by Theorem 2.13 in \cite{JM14} if $R$ is idempotent, $C$ is prime if it is QC and irreducible, that is if it can not be obtained as the intersection of two strictly larger congruences.
    \item {\it radical} if $C$ is equal to the intersection of all prime congruences containing it, which we denote by $Rad(C)$.
\end{itemize}
\end{definition}

\begin{proposition}[Proposition 2.10 in \cite{JM14}]\label{prop: 2.10}
Let $A$ be a $\mb$-algebra. If $A$ is totally ordered then the trivial congruence of $A$ is prime if and only if $A$ is cancellative. 
\end{proposition}

We will often refer to the following immediate consequences of the above:
\begin{corollary}\label{cor: prop_prime}
Let $A$ be a $\mb$-algebra and $C$ a congruence on $A$. 
\begin{enumerate}
    \item The congruence $C$ is prime if and only if $A/C$ is cancellative and totally ordered.
    \item In the special case when every element of $A$ is a sum of invertible elements, $C$ is prime if and only if $A/C$ is totally ordered. 
    \item Each prime congruence of $A$ is generated by pairs of the form $(a, a+b),$ for $a, b \in A$. 
\end{enumerate}
\end{corollary}

\begin{proof}
Applying Proposition~\ref{prop: 2.10}, it remains to be seen that the quotient of a prime congruence $P$ is totally ordered. Indeed for arbitrary elements $a,b \in A$, the twisted product $$(a + b, a)(a + b, b) = (a^2 + b^2 + ab, a^2 + b^2 + ab),$$ lies in the diagonal, hence every prime contains at least one of $(a + b, a)$ and $(a + b, b)$, implying that in the quotient we have $b \geq a$ or $a \geq b$. 

For part (2), consider that if every element of $A$ is sum of invertibles, then the same holds in the quotient $A/C$. Moreover since $A/C$ is totally ordered every sum equals one of its summands, hence every element of $A/C$ is invertible. In particular $A/C$ is a semifield and the claim follows from part (1).

Part (3) follows from (1) and the fact that $(a+b, a) \in C \iff a \geq b$ with respect to the order defined by $C$.
\end{proof}

In view of the previous result we introduce the following notation. 
\begin{notation}
Let $A$ be semiring and let $P$ be a prime on $A$. For two elements $r_1, r_2 \in A$ we say that $r_1 \leq_P r_2$ (resp. $r_1 <_P r_2$, resp. $r_1 \sim_P r_2$) whenever $\overline{r_1} \leq \overline{r_2}$ (resp. $\overline{r_1} <  \overline{r_2}$, resp. $\overline{r_1} = \overline{r_2}$) in $A/P$. Here by $\overline{r}$ we mean the image of $r$ in $A/P$.
\end{notation}

\begin{lemma}\label{lem: primes_contain_primes}
Let $A$ be a $\mb$-algebra and let every element of $A$ be a sum of invertibles. Let $P$ and $Q$ be congruences on $A$ such that $P \subseteq Q$. Then if $P$ is prime then so is $Q$.
\end{lemma}

\begin{proof}
Since $P$ is prime then $A/P$ is totally ordered. However, via the semiring morphism $A/P \twoheadrightarrow A/Q$ so is $A/Q$. Now Corollary~\ref{cor: prop_prime} (2) implies that $Q$ is prime. 
\end{proof}

In this paper the only $\mb$-algebras we will consider will be $\tPoly$, $\tLaur$ and their quotients.

\begin{definition}\label{def:geo-prime}
We call a prime congruence $P$ of $\tLaur$ {\it a geometric prime} whenever $\tLaur/P \cong \TT$.
In particular, the points of $\RR^n$ are in bijection with the geometric prime congruences of $\tLaur$. 
We call a prime congruence $P$ of $\tLaur$ {\it a minimal prime} if there are no prime congruences contained in it. We define geometric and minimal primes on $\tPoly$ analogously.
\end{definition}

\begin{proposition}\label{prop: min_prime}
   Let $P$ be a minimal prime congruence on $\tPoly$ or $\tLaur$ then each equivalence class of $P$ contains exactly one monomial.
\end{proposition}

\begin{proof}
    Follows directly form \cite[Theorem 4.14(b)]{JM14}.
\end{proof}

In \cite[Section 4]{JM14} a complete description of the prime congruences of $\tLaur$ and $\tPoly$ was given. In the quotients each element can be represented by a monomial $t\vecx^\vecu$, hence primes correspond to semigroup orderings on quotients of the multiplicative semigroup of monomials. Similarly to term orders for polynomial rings, each such ordering can be defined by a matrix $U$ with real entries and $n+1$ columns, setting $t_1\vecx^{\pmb{u_1}} \geq t_2\vecx^{\pmb{u_2}}$ if and only if the first non zero element of $U((t_1, \pmb{u_1}) - (t_2, \pmb{u_2}))$ is positive. As $U$ must respect the order on the elements of $\TT$, the first column of $U$ is either zero or its first non-zero entry is positive. A matrix that satisfies this condition and that each of its rows are non redundant in the sense that they further refine the ordering is called \textit{admissible}. It was shown in \cite[Theorem 1.1]{JM14} that every prime congruence $P$ of $\tLaur$ is defined by an admissible matrix $U$ with rank $r(U) = \dim(\tLaur / P)$. Moreover that the primes of $\tPoly$ which do not identify any variables with $0_{\TT}$ are in bijection with the primes of $\tLaur$, and the primes that do identify any variables with $0_{\TT}$ hence correspond to a prime of a polynomial semiring in less variables.

%Moreover that the primes of $\tPoly$ with trivial kernel \red{---we haven't defined kernel of a prime ---- } are in bijection with the primes of $\tLaur$, and the primes with non-trivial kernel identify a subset of the variables with $0$ and hence correspond to a prime of a polynomial semiring in less variables.

Every point of $P \in \RR^n$ defines a congruence by identifying polynomials that are equal on that point. These congruences are prime, and the congruence corresponding to the point $(a_1, \dots, a_n) \in \RR^n$ is defined by the matrix $U = (1, a_1, \dots, a_n)$. These are exactly the geometric primes, and for a polynomial $f$  and the geometric prime congruence $P$ we write $f(P)$ for the value of $f$ on the corresponding point.

If $f = \sum_{u \in \mathbb{Z}^n} {a_u}\vecx^\vecu \in \tLaur$, then we say that $f$ \textit{tropically vanishes} at a point $P \in \RR^n$ if $f$ attains its maximum at $P$ at least twice, where $f(P) = \text{max}\{ a_u + u\cdot P : u\in supp(f) \}$. Note that $+$ denotes the usual addition of real numbers. With some care this definition can be extended to polynomials of $\tPoly$ as shown in \cite{MS}.

%In \cite{JM14} we show that we can identify the points of $\rma^n$ with the geometric prime congruences on $\tPoly$. 
Now let $I$ be an ideal of $\tLaur$, $H \subseteq \RR^n$, and $C$ is a congruence on $\tLaur$. In this paper we will consider the following sets:
\begin{equation*}
\begin{split}
\VV(I) &= \{P \in \RR^n : \forall f \in I, f \text{ tropically vanishes at } P\} \\
 & = \{P \text{ geometric prime congruence} : \forall f \in I, bend(f) \subseteq P\},  \\
\II(H) &= \{f : f \text{ tropically vanishes on } H\} \\
&= \{f : bend(f) \in P,\ \forall \text{ prime congruences } P \in H\} \\
\VC &= \{ a \in \RR^n : f(a) = g(a), \forall (f,g) \in C\} \\
&= \{P \text{ geometric prime congruence } : P \supseteq C\}
%\text{ and } \\
%\IC &= \{ f \in \tPoly : f \text{ attains its maximum twice with respect to the order }\\ &\ \ \ \ \text{ of the prime congruence }C \}
\end{split}
\end{equation*}

We refer to the first of these sets as the \textit{tropical variety of $I$}, or \textit{tropical vanishing locus of $I$}. We can similarly define these sets for ideals and congruences of $\tPoly$ and subsets and points on $\TT^n$. 

We will often want to know if there is geometric prime congruence containing a prime congruence $P$ on $\tLaur$ or $\tPoly$. If $P$ has a geometric prime lying over it then $P$ does not collapse the coefficient semifield to $\mb$. This condition is also equivalent to having %a non-zero entry in the first column of the defining matrix of $P$.
the (1,1) entry of the defining matrix of $P$ not equal to 0. 

\begin{remark}\label{rem: V(prime_cong)}
The variety $\VV(P) \subseteq \RR^n$ of a prime congruence $P$ contains at most one point. Note that $\VV(P)$ is empty when $P$ is not contained in a geometric prime congruence. Assume for contradiction that there are two distinct points $a, b$ in $\VV(P)$. Consider the congruence $C = \{ (f,g)\ :\ f(a)=g(a) \text{ and } f(b) = g(b)\}$. Notice that $P \subseteq C$ and that there is a surjective semiring morphism $\tLaur/P \rightarrow \tLaur/C$. Moreover, since $\tPoly/P$ is totally ordered, so is $\tLaur/C$. However, this cannot be the case since $\tLaur/C$ is isomorphic to $\rma \times \rma$ via the map $f \mapsto (f(a), f(b))$. The same is true for the varieties of prime congruences on $\tPoly$.
\end{remark}

\begin{example}\label{ex:var-geom-point}
    Let $P$ be a prime congruence with defining matrix $U$, such that the first row of $U$ is $(1, a_1, \dots, a_n)$, then $\VV(P) = (a_1, \dots, a_n)$.
\end{example}

% In view of Example~\ref{ex:var-geom-point} we will sometimes refer to geometric primes as \textit{geometric points}. - i promise you we will do no such thing :)

%------------------ Dimension prelims ----------------------------

One can define (Krull) dimension for semirings, analogously to the Krull dimension of rings, measured by the number of strict inclusions in the longest chain of prime ideals. As shown in \cite{AA94}, with this definition, the dimension of any polynomial semiring is infinite. In \cite{JM14} we propose the following definition for the Krull dimension of a semifield.

\begin{definition}\label{def:Krull}
Let $R$ be a $\mb$-algebra. The \emph{Krull dimension} of $R$ is the number of strict inclusions in the longest chain of prime congruences in $R$.
\end{definition}

%In the case when $R$ is the semiring $\tPoly$ we can say even more. Since every prime congruence is defined by an admissible matrix $U$, we say that $P$ is \emph{a prime congruence of rank} $r$ whenever $r$ is the rank of $U$, denoted $r(U)$. We also show \cite[Proposition~4.13]{JM14} that the Krull dimension of the quotient $\tPoly/P$ is equal to $r(U)$. \\

Classically, the dimension of an irreducible affine variety is equal to the Krull dimension of its coordinate ring. Tropically, we need to replace the notion of the coordinate ring with the ``coordinate semiring'' of the affine tropical scheme associated to this tropical variety as defined in \cite{GG13}.
\begin{definition}
If $X$ is a tropical variety defined by an ideal $I \in \tPoly$, the coordinate semiring of the tropical scheme associated to $X$ is $\tPoly/Bend(I)$.   
\end{definition}

\newpage
\section{Results}
\subsection{Tropical ideals}

\begin{definition}\label{def: trop_ideals}
An ideal $I \subseteq \tPoly$ is a \textit{tropical ideal} if for each degree $d \geq 0$ the set $I_{\leq d}$ of polynomials in $I$ of degree at most $d$ is a tropical linear space, or equivalently, $I_{\leq d}$ is the set of vectors of a valuated matroid. More precisely, the ideal $I$ satisfies a monomial elimination axiom which can be stated in the following way:\\
For any $f, g \in I_{\leq d}$ and any monomial ${\pmb x}^{\pmb u}$, which appears in both $f$ and $g$ with the same ($\neq 0_{\rma}$) coefficient then there exists a polynomial $h \in I_{\leq d}$, such that the coefficient $c_u(h)$ of ${\pmb x}^{\pmb u}$ in $h$ is $0_{\rma}$ and for the coefficient $c_v(h)$ of any other monomial ${\pmb x}^{\pmb v}$ of $h$ we have that $c_u(h) \leq \max (c_v(f), c_v(g))$ with equality holding whenever $c_v(f) \neq c_v(g)$.
\end{definition}

\begin{example}\label{ex:realizable}
Consider a polynomial $f \in K[x_1, \dots, x_n]$, where $K$ is a valued field. We will denote by $trop(f)$ the polynomial in $\tPoly$ obtained from $f$ by valuating all coefficients of $f$ and replacing the operations multiplication and addition with their tropical counterparts. Let $J \subseteq K[x_1, \dots, x_n]$, then the ideal $trop(J) = \{ trop(f) \ : \ f\in J \}$ is a \emph{realizable} tropical ideal. 
\end{example}
%\noindent{\color{red} define tropical ideal in $\tLaur$?}

%here is a one to one correspondence between prime tropical ideals of $\tLaur$ and prime tropical ideals of $\tPoly$ not containing a monomial

% \orange{
% \begin{proposition}\label{prop: trop_ideals_corresp}
% There is a bijection between prime tropical ideals of $\tLaur$ and prime tropical ideals of $\tPoly$ not containing a monomial via the maps
%    $$ I \to I\cap \tPoly \quad \text{and} \quad J \to J\tLaur.$$
% \end{proposition}
% 

\begin{proposition}\label{prop: trop_ideals_corresp}
%  There is a one to one correspondence between prime tropical ideals of $\tLaur$ and prime tropical ideals of $\tPoly$ not containing a monomial given by the following maps:
% \begin{align*} 
%     \{\text{Prime tropical ideals of } \tLaur \} &\to \{ \text{Prime tropical ideals of } \tPoly \text{ not containing a monomial}\}\\
%                                     I &\mapsto I\cap \tPoly \\
%                                     J\tLaur &\mapsfrom J
% \end{align*}
There exits a bijection $\varphi$ from the set of prime tropical ideals of $\tLaur$ to set of prime tropical ideals of $\tPoly$, defined as $\varphi(I) = I\cap \tPoly$ and $\varphi^{-1}(J) = J\tLaur$. 
\end{proposition}

\begin{proof}
Let $I \subseteq \tLaur$ be a tropical prime ideal and $J = I \cap \tPoly$. Note that $J$ is a proper ideal and can not contain a monomial. We will show that $J$ is also prime and tropical. Assume it is not prime, that is, there are $f, g \in \tPoly$ such that $fg \in J$ but neither $f$ nor $g$ is in $J$. But since $J \subseteq I$ then $fg \in I$ and since $I$ is prime then without loss of generality $f \in I$. However, since $f \in \tPoly$ so $f \in J$, contradicting the assumption that $J$ is not prime. \par
%Now assume that $J$ is not tropical. Then there exist $f, g \in J$ such that they don't satisfy the "cancelling" condition in Definition~\ref{def: trop_ideals}. However, since $J \subseteq I$, and $I$ is tropical then $f, g$ satisfy Definition~\ref{def: trop_ideals} 
Since $I$ is a tropical ideal, for any $f, g \in J \subseteq I$ there exists some $h \in I$ which satisfy the property in Definition~\ref{def: trop_ideals}. Since $f, g$ are polynomials, so is $h$ in the definition. So $J$ satisfies the conditions of Definition~\ref{def: trop_ideals} and hence $J$ is tropical.  \par
Now consider $J \subseteq  \tPoly$ and let $I = J\tLaur$. Note that $I$ is a proper ideal if and only if $J$ does not contain a monomial. Let $fg \in I$, then there exists a monomial $m \in \tPoly$ such that $mf$ and $mg$ are both polynomials. Then $mfmg \in I$ but it is also in $J$. Now without loss of generality let $mf \in J$ and since $m \not\in J$ by assumption, then $f \in J$ and so $f \in I$. Similarly, $I$ is tropical, since the circuits of $I$ are circuits of $J$ multiplied by monomials. 
\end{proof}

\subsection{Relating prime ideals and prime congruences}
$ \ $\par

In view of Proposition~\ref{prop: trop_ideals_corresp} we will state all results for the Laurent polynomials. We describe a relation between ideals and congruences in $\tLaur$. Let $C$ be a congruence on $\tLaur$. Consider the set
%Denote by $\polyFromBend{C}$ the set of polynomials whose bend relations lie in $C$.
$$\polyFromBend{C} = \{f \in \tLaur : bend(f) \subseteq C\}.$$
One can define this similarly in the $\tPoly$ case. 

%It is easy to check that $\polyFromBend{C}$ is an ideal. 

    \begin{proposition}
        Let $C$ be any congruence on $\tLaur$ then the set $\polyFromBend{C}$ is an ideal. 
    \end{proposition}

\begin{proof}
    Let $f, g \in \polyFromBend{C}$, i.e., $bend(f), bend(g) \subseteq C$. Consider the pair $(f+g, (f+g)_{\hat \imath})$ for each $i \in supp(f+g)$. If $i$ is in the support of both summands, then $(f+g, (f+g)_{\hat \imath}) = (f, f_{\hat \imath}) + (g, g_{\hat \imath}) \in C$, since $bend(f), bend(g) \subseteq C$. In the case when $i$ is only in the support of one of the summands, say $f$, then $(f+g, (f+g)_{\hat \imath}) = (f, f_{\hat \imath}) + (g, g) \in C$. This shows that $bend(f+g) \subseteq C$ and thus $f+g \in \polyFromBend{C}$.
    
    Let $f \in \polyFromBend{C}$ and $h \in \tLaur$. We will show that $fh \in \polyFromBend{C}$, by showing that $bend(fh) \subseteq C$. Consider the pair $(fh, (fh)_{\hat k})$ for each $k \in supp(fh)$. Let $k = (i, j)$, i.e., $(fh)_k = f_ih_j$. If $f_ih_j \neq af_sh_r$, for all $(i,j) \neq (s,r)$ and $a\in \TT$, then by direct computation we see that $(fh, (fh)_{\hat k}) = (f, f_{\hat\imath})(h, h_{\hat\jmath})$.

    Now let $f_ig_j = af_sg_r$ for some $(i,j) \neq (s,r)$ and $a\in \TT$. Assume without loss of generality that the coefficient of $f_sg_r$ is bigger than that of $f_ig_j$. %; in particular $f_ig_j + f_sg_r = f_sg_r$. 
    Then we can write $(fh, (fh)_{\hat k}) = (f, f_{\hat s})(h_{\hat\jmath}, 0) + (d, d),$ where $d$ is the polynomial $$d = \sum_{\substack{l\in supp(f), \\l\neq i}}f_lh_j + \sum_{\substack{l\in supp(h), \\l\neq r}}f_sh_l.$$ Now $(f, f_{\hat s})(h_{\hat\jmath}, 0)$ is in $C$ by Proposition~\ref{prop: closedTwProd}. The pair $(d,d)$ also in $C$ since the diagonal is contained in every congruence and so we conclude that $(fh, (fh)_{\hat k}) \in C$.
\end{proof}

If $I$ an ideal of $\tLaur$ and $C = Bend(I)$ then $\polyFromBend{C} \supseteq I$. In fact, equality rarely holds. This is illustrated in the following example:

\begin{example}
Let $I$ be the ideal of $\mb[x^{\pm 1},y^{\pm 1}]$ generated by the set $\{x+y, x+z\}$. Observe that $bend(y+z) \subseteq Bend(I)$ and so $y+z \in \polyFromBend{Bend(I)}$ but $y+z \not\in I$. 
\end{example}

This example justifies the following definition.
\begin{definition}\label{def:closed_ideal}
An ideal $I$ of $\tLaur$ or $\tPoly$ is $\textit{closed}$ if $\polyFromBend{Bend(I)} = I$.
\end{definition}

% \red{-- next prop should be removed -- }
% \begin{proposition}\label{prop: closed_ideals_corresp}
% Let $I \subseteq \tLaur$ be an ideal. I is closed if and only if $I \cap \tPoly$ is.
% \end{proposition}

% \begin{proof}
% Note that an ideal $J$ is closed if and only if, whenever for a polynomial $f$ we have $bend(f) \subseteq Bend(I)$ we also have $f \in I$. First let us assume that $I$ is a $\closed$ ideal and $f \in I \cap \tPoly$, such that $bend(f) \subseteq Bend(I \cap \tPoly)$. Then we also have $bend(f) \subseteq Bend(I)$ and hence $f \in I \cap \tPoly$, showing that $I \cap \tPoly$ is closed.
% For the other direction if $I \cap \tPoly$ is $\closed$ and $bend(f) \subseteq Bend(I)$ then $bend(f)$ is generated by the bend relations of finitely many elements $g_1,\ldots,g_k \in I$. Pick $m$ to be a monomial such that all of $mf, mg_1, \ldots, mg_k$ are elements of $\tPoly$. It follows that $bend(mf)$ is generated by $bend(mg_1),\ldots,bend(mg_k)$ and since $I \cap \tPoly$ is $\closed$ this implies that $mf \in I \cap \tPoly$ and so $f \in I$.
% \end{proof}

\begin{proposition}
The intersection of $\iclosed$ ideals is a $\iclosed$ ideal.
%Let $I_i$ be ideals which satisfy the property $\polyFromBend{Bend(I_i)} = I_i$, then so does their intersection.
\end{proposition}

\begin{proof}
For $\closed$ ideals $J, L$ assume that $f \in \polyFromBend{Bend(J \cap L)}$. Then we have $f \in \polyFromBend{Bend(J)}$ and $f \in \polyFromBend{Bend(L)}$. Since $\polyFromBend{Bend(J)} = J$ and $\polyFromBend{Bend(L)} = L$ we conclude that $f \in J \cap L$.
\end{proof}

\begin{example}
 $ $\begin{enumerate}
    \item All tropical ideals are $\closed$ by \cite[Theorem 1.1]{MR14}.
    \item While the intersection of tropical ideals is not necessarily tropical, it is a $\closed$ ideal.
    \item The unique maximal ideal of $\tPoly$ consisting of every non-constant polynomial is $\closed$. 
\end{enumerate}
\end{example}

\begin{remark}\label{rem:closed}
To every congruence $C$ one can associate at most one $\closed$ ideal $I$ such that $C=Bend(I)$ and this $\closed$ ideal is $\polyFromBend{C}$ by definition.
\end{remark}

\begin{example}\label{ex:I_P}
%\begin{enumerate}
   % \item %If $P$ is a geometric prime congruence, then $\polyFromBend{P}$ is a tropical ideal containing the polynomials that attain their maximum at least twice at $P$. %corresponding to the point $P$.
    
    Let $P$ be a geometric prime corresponding to the point $(b_1,\ldots,b_n)\in\rma^n$. Fix a field $k$ with a surjective valuation $v:k\to\rma$ and pick $a_1,\ldots,a_n\in k$ such that $v(a_i)=b_i$. Let $\mathfrak{m}$ be the maximal ideal of $k[x_1\ldots,x_n]$ corresponding to the point $(a_1,\ldots,a_n)\in k^n$. Now $\polyFromBend{P}$ is the tropicalization of $\mathfrak{m}$, and is therefore a tropical ideal. In particular, $\polyFromBend{P}$ consists of those polynomials which attain their maximum at least twice at $P$.
    %More precisely, let $k$ be a valued field with valuation $v:k \to \mathbf{T}$. Let $\mathfrak{m}$ be a maximal ideal of $k[x_1, \dots, x_n]$. The variety of $\mathfrak{m}$ is a point, say $(a_1, \dots, a_n) \in (k^\times)^n$ and its tropicalization is the point $(v(a_1), \dots, v(a_n)) \in \mathbb{R}^n$. Let $P$ be the prime congruence on $\tLaur$ with defining matrix $(1, v(a_1), \dots, v(a_n))$. Then the ideal $\polyFromBend{P}$ is the tropicalization of $\mathfrak{m}$. Note that $Bend(I_P) = \left<x_i \sim v(a_i) : 1\leq i \leq n\right>.$
%    \item When $P$ is a minimal prime congruence on $\tLaur$ %or $\tPoly$ 
    then by Proposition~\ref{prop: min_prime} each equivalence class of the congruence $P$ contains a single term, i.e., no polynomial will attain its maximum twice with respect to the order induced by $P$, hence $\polyFromBend{P}$ is the zero ideal. 
%\end{enumerate}

\end{example}

We are ready to state the first main result of this paper. 
\begin{theorem}\label{prop: bendPrime-iff-Iprime}
\begin{enumerate}
    \item[(i)]  Every closed prime ideal in $\tLaur$ is of the form $\polyFromBend{P}$ for a prime congruence P. 
    \item[(ii)] $\polyFromBend{P}$ is the zero ideal if and only if $P$ is a minimal prime.
    \item[(iii)]  The maps $I \rightarrow Bend(I)$ and $P \rightarrow \polyFromBend{P}$ establish a one to one correspondence between the non-zero closed prime ideals and the non-minimal prime congruences of $\tLaur$.
    \item[(iv)]  The closed prime ideal $\polyFromBend{P}$ is a tropical ideal if and only if P is a geometric prime.
\end{enumerate}
\end{theorem}

The rest of the section is dedicated to the technical statements needed to prove this result. 

\begin{proposition}\label{prop: bendPrime-Iprime}
The ideal $\polyFromBend{P} \subseteq \tLaur$ is prime whenever $P$ is a prime congruence on $\tLaur$.
\end{proposition}

\begin{proof}
If $P$ %restricts to a minimal prime on $\mb[x_1, \ldots, x_n]$, 
is a minimal prime congruence on $\tLaur$, then as noted in Example~\ref{ex:I_P} the ideal $\polyFromBend{P}$ is the zero ideal so the statement is true because the zero ideal in $\tLaur$ is prime. Let $f, g $ be two polynomials, $f = \sum f_l$ and $g = \sum g_q$, such that $f, g\not\in \polyFromBend{P}$ but $fg \in \polyFromBend{P}$ and so $bend(fg) \subseteq Bend(\polyFromBend{P}) \subseteq P$. If in the formal expression $fg_{\hat \jmath}+f_{\hat \imath}g$ none of the terms $f_ig_j$ and $f_sg_r$ for all pairs $(i,j) \neq (s,r)$ have the same exponent vector, then we can write
\begin{equation}\label{eq:prodTw}
    (f, f_{\hat \imath})(g, g_{\hat \jmath}) = (fg + f_{\hat \imath}g_{\hat \jmath}, fg_{\hat \jmath}+f_{\hat \imath}g) = (fg, fg_{\hat \jmath}+f_{\hat \imath}g) = (fg, (fg)_{\hat k}),
\end{equation}
where $i$ and $j$ are in the support of $f$ and $g$ respectively, $(fg)_k = f_ig_j$, and $f_{\hat \imath}$ denotes $f$ with the term $i$ omitted. In particular, since $(fg, (fg)_{\hat k}) \in bend(fg) \subseteq P$ and $P$ is prime, either $(f, f_{\hat \imath}) \in P$ or $(g, g_{\hat \jmath})\in P$. 

Now assume that for two pairs $(i,j) \neq (s,r)$ the terms $f_ig_j$ and $f_sg_r$ have the same exponent vector. Without loss of generality let $f_sg_r \geq f_ig_j =:d$. Then $$(fg, (fg)_{\hat k}) + (d,d) = (f, f_{\hat \imath})(g, g_{\hat \jmath}).$$ Since both summands on the left hands side are in $P$ then the sum is in $P$, and so either $(f, f_{\hat \imath}) \in P$ or $(g, g_{\hat \jmath})\in P$. Note that when $f_sg_r = f_ig_j$, then $(f, f_{\hat \imath})(g, g_{\hat \jmath}) = (fg, fg)$ which is also in $P$.

Since this holds for each pair $i, j$, it follows that either $f$ or $g$ has all their bend relations in $P$, or equivalently $f \in \polyFromBend{P}$ or $g \in \polyFromBend{P}$.
\end{proof}

The following example shows that not all prime ideals are of the form $\polyFromBend{P}$, in particular, not all prime ideals are $\closed$:
\begin{example}
%\noindent{\color{red}Non-homogeneous with respect to grading. -----------}
Consider the ideal $I \subseteq \TT[x^{\pm 1},y^{\pm 1}]$ that consists of all polynomials that are not homogeneous with respect to $x$. This is a prime ideal since the product of homogeneous polynomials is homogeneous. Consider $f, g \in I$, where $f = x+1$ and $g = xy+1$ and so $bend(f) = (x,1)$ and $bend(g) = (xy,1)$ are elements of $Bend(I)$. Since $Bend(I)$ is a congruence, in particular it is transitively closed, we get that $(x, xy) \in Bend(I)$. However, $x + xy \not\in I$ since it is homogeneous with respect to $x$.
\end{example}

\begin{proposition}\label{prop: PeqBendIP}
Let $P$ be a non-minimal prime congruence of $\tLaur$,
then $P = Bend(\polyFromBend{P})$.
\end{proposition}

\begin{proof}
The inclusion $Bend(\polyFromBend{P}) \subseteq P$ is always true. For the other direction, first note that $P$ is generated by pairs of the form $(m_1 + m_2, m_1)$, where $m_1$ and $m_2$ are terms, since every element of its quotient can be represented by a term.
Since $P$ is not a minimal prime there must be distinct terms $a, b$ such that $a \neq b$ and $(a,b) \in P$. Now for any $m_1, m_2$, such that $(m_1 + m_2, m_1) \in P$, equivalently written $m_1+m_2 \sim_P m_1$, we have:

\begin{equation}\label{eq: IP1}
am_1 + am_2 \sim_P am_1 + bm_1 \sim_P am_2 + bm_1 \Rightarrow am_1 + am_2 + bm_1 \in \polyFromBend{P}\end{equation}
and 
\begin{equation}\label{eq: IP2}
am_1 \sim_P bm_1 \Rightarrow am_1 + bm_1 \in \polyFromBend{P}.
\end{equation}
Taking the bend relations of the polynomials in Equations~(\ref{eq: IP1}) and (\ref{eq: IP2}) we obtain:
\begin{equation}\label{eq: IP3} 
am_1 + am_2 \sim_P am_1 + bm_1 \sim_P am_1 \in Bend(\polyFromBend{P}).
\end{equation}
Since $a \in \tLaur$ it is invertible and we can divide the pairs in Equation~(\ref{eq: IP3}) by $a$ and we obtain that $(m_1 + m_2, m_1) \in Bend(\polyFromBend{P}).$ 
\end{proof}

\begin{proposition}\label{prop: prime_has_binom}
Every non-zero prime ideal $I$ of $\tLaur$ contains a binomial.
\end{proposition}
\begin{proof}
Consider the polynomial $f \in I$, and write it as $f = a + b + c$, where $a, b$ are terms and $c = m_1 + \dots + m_n$ and $m_i$ are terms. If $n = 0$ then we are done. Assume $n > 0$ and observe that $(a+b+c)(ab + bc + ac) = (a + b)(a+c)(b+c)$ and since $I$ is prime then one of $a + b, a+c$ and $b+c$ is in $I$. If $a+b \in I$ we are done. Otherwise, without loss of generality let $a+c \in I$. We can now apply the same argument to $a + c$, noticing it has one less term than $f$. %$a+c = a + b_1 + c_1$, where $b_1 = m_n$ and $c = m_1 + \dots + m_{n-1}$
Iterating the procedure we reduce to the case $n = 0$. 
\end{proof}

\begin{proposition}\label{prop: Iprime-bendPrime}
Let $I$ be a non-zero prime ideal of $\tLaur$. Then $Bend(I)$ is a prime congruence. 
\end{proposition}

\begin{proof}
%In view of Proposition~\ref{prop: closed_ideals_corresp} assume $I \subseteq \tLaur$. 
To show that $Bend(I)$ is a prime congruence it is enough to check that the quotient by it is totally ordered in view of Proposition~\ref{cor: prop_prime} (2). In other words, we need to see that for any two terms $m_1$ and $m_2$ one of the pairs $(m_1 + m_2, m_1)$ and $(m_1 + m_2, m_2)$ is in $Bend(I)$.

Let $f = a+b$ be a binomial in $I$, which exists by Proposition~\ref{prop: prime_has_binom} and hence $(a, b) \in Bend(I)$. Now let $m_1$ and $m_2$ be two terms and consider the product
$$((a+b)m_1 + am_2)((a+b)m_2 + b m_1) = (a+b)((a+b)m_1m_2 + b m_1^2 + a m_2^2) \in I.$$
Since $I$ is prime, then one of $(a+b)m_1 + am_2$ and $(a+b)m_2 + b m_1$ is in $I$. Without loss of generality $(a+b)m_1 + am_2 \in I$ and so $(am_1 + am_2, am_1 + bm_1) \in Bend(I)$. But since $(a,b) \in Bend(I)$ by assumption, we have that $(am_1 + bm_1, am_1) \in Bend(I)$ and we conclude by transitivity that $(am_1 + am_2, am_1) \in Bend(I)$. Since $a$ is a non-zero Laurent term it is invertible and so $(m_1 + m_2, m_1) \in Bend(I)$.
%As in the proof of Proposition~\ref{prop: PeqBendIP} we can conclude that $(m_1+m_2, m_1) \in Bend(I)$.
\end{proof}

The converse is not necessarily true as shown in the following example.
\begin{example} 
Consider the ideal $I \subseteq \rma[x^{\pm 1}, y^{\pm 1}]$ given as the intersection of the set of all polynomials not homogeneous in $x$ and the set of all polynomials not homogeneous in $y$. Observe that $I$ is not a prime ideal because $(x+xy)(y+xy) \in I$ but neither of the factors is an element of $I$.
On the other hand, since $f = x+y$ and $g = x+y^2$ are both elements of $I$, then $(x, y), (y,1) \in Bend(I)$. Thus $$\rma[x^{\pm 1}, y^{\pm 1}]/Bend(I) \cong \rma,$$ and since $\rma$ is a totally ordered semifield by Corollary~\ref{cor: prop_prime} we conclude that $Bend(I)$ is a prime congruence.
\end{example}

We can now give a characterization of the $\closed$ prime ideals of $\tLaur$ and identify the tropical ideals among them. 

\begin{lemma}\label{lem: trop-condition}
    %Let $P$ be a prime congruence on $\tLaur$ and let $m_1 = c_1{\pmb x_1}^{\pmb u_1}$, and $m_2 = c_1{\pmb x_2}^{\pmb u_2}$ in $\tLaur$, such that ${\pmb u_1} \neq {\pmb u_2}$ and in $\tLaur/P$ we have that $m_1 = m_2$ and $cm_1 > 1$  for all $c \in \RR$, then $\polyFromBend{P}$ is not a tropical ideal.
    Let $P$ be a prime congruence on $\tLaur$ and let $m_1 = c_1{\pmb x_1}^{\pmb u_1}$, and $m_2 = c_1{\pmb x_2}^{\pmb u_2}$ be monomials in $\tLaur$, such that ${\pmb u_1} \neq {\pmb u_2}$. If we have that $m_1 \sim_P m_2$ and $cm_1 >_P 1$  for all $c \in \RR$, then $\polyFromBend{P}$ is not a tropical ideal.
\end{lemma}

\begin{proof}
    Indeed, assume for contradiction that $\polyFromBend{P}$ is tropical. Let $m_3$ be a term such that in we have $1 >_P m_3$. Note that such $m_3$ exists -- for example $m_1^{-1}$ -- and it must have a distinct exponent from $m_1$ and $m_2$. Now the polynomials $m_1 + m_2 + m_3$ and $m_1 + m_2 + m_3^2$
   attain their maximum twice at $P$ hence they are in $\polyFromBend{P}$. By the monomial elimination axiom of tropical ideals we can eliminate $m_1$, but the resulting polynomial $c m_2 + m_3 + m_3^2$ can not attain its maximum twice at $P$ as $m_3 >_P m_3^2$ and we either have $c = 0$ or $c m_2 >_P m_3$.     
\end{proof}

\begin{remark}\label{rmk: higher-rank-primes}
    A non-minimal prime congruence $P$ whose matrix has more than one row is not finitely generated. Moreover, the variety of $P$ is not determined by finitely many pairs in $P$ as shown in \cite[Example 6.3.4]{M16}. Note that for any $P$, by definition $\VV(P) = \VV(\polyFromBend{P})$. Thus the ideal $\polyFromBend{P}$ corresponding to a prime $P$ with rank greater than one is not tropical because it will not have a finite tropical basis. 
\end{remark}

\begin{proof}[Proof of Theorem \ref{prop: bendPrime-iff-Iprime}]
For (i) if $I$ is a non-zero closed prime then by Proposition~\ref{prop: Iprime-bendPrime} $Bend(I)$ is a prime congruence, and since $I$ is closed we have that $I = \polyFromBend{Bend(I)}$. As it was explained in Example~\ref{ex:I_P} the zero ideal can be written as $\polyFromBend{P}$ for any minimal prime $P$. This also proves one direction of (ii) and the other follows from Proposition~\ref{prop: PeqBendIP}.
For (iii) again by Proposition~\ref{prop: Iprime-bendPrime} we have that for a non-zero closed prime $I$ the congruence $Bend(I)$ is prime and for a non-minimal prime congruence $P$ the ideal $\polyFromBend{P}$ is prime by Proposition~\ref{prop: bendPrime-Iprime}, non-zero by (ii) and closed by Remark~\ref{rem:closed}. Moreover, for a closed ideal $I$ we have that $I = \polyFromBend{Bend(I)}$ and for a non-minimal prime $P$ by Proposition~\ref{prop: PeqBendIP} we have that $P = Bend(\polyFromBend{P})$, showing that the two maps are indeed inverses of each other.

For (iv) we have already seen in Example~\ref{ex:I_P} (1) that if $P$ is a geometric prime then $\polyFromBend{P}$ is tropical. We will show that if $P$ is not geometric, then $\polyFromBend{P}$ is not tropical. If the defining matrix $U$ of $P$ consists of multiple rows, then by Remark~\ref{rmk: higher-rank-primes} the ideal $\polyFromBend{P}$ is not tropical. Remains to consider non-geometric primes whose matrices have a single row. Recall that a single row matrix corresponds to a geometric prime if the $(1, 1)$ entry is non-zero and let us assume that the $(1,1)$ entry of $U$ is zero, or equivalently that $P$ identifies the elements of $\TT$ with $1$. By Proposition~\ref{prop: prime_has_binom} there is a binomial $m_1 + m_2 \in \polyFromBend{P}$. By multiplying both with a monomial whose exponent vector is a sufficiently large multiple of the first row of $U$ we can assume that $m_1 \sim_P m_2 >_P 1$, and as we assumed that $P$ identifies $\TT$ with $1$ we also have that $c m_1 >_P 1$ for all $c \in \RR$. Hence we can apply Lemma~\ref{lem: trop-condition}, and conclude that $\polyFromBend{P}$ is not a tropical ideal.
\end{proof}

\subsection{Varieties of congruences and ideals}
$ \ $\par\medskip
% \noindent Combining Proposition~\ref{prop: bendPrime-Iprime} and Proposition~\ref{prop: Iprime-bendPrime} we can conclude the following.
We use the results of the previous section to describe the varieties of prime ideals.

\begin{proposition}
Let $I$ be a prime ideal of $\tLaur$ or $\tPoly$ then $\VV(I)$ is at most a point.
\end{proposition}

\begin{proof}
For the case of $\tLaur$ from Proposition~\ref{prop: Iprime-bendPrime} we know that $Bend(I)$ is prime and so $\VV(Bend(I))$ consists of at most a point by Remark~\ref{rem: V(prime_cong)}. The theorem follows as $\VV(I) = \VV(Bend(I))$ which is true by definition of the sets. For $\tPoly$, if $I$ is a prime ideal containing a monomial, then it also contains at least one of the variables appearing in that monomial hence its tropical vanishing locus is in bijection with that of a tropical prime ideal in less variables. If $I$ does not contain monomials then the statement follows from Proposition~\ref{prop: trop_ideals_corresp} and the $\tLaur$ case.
\end{proof}

\begin{example} We give an example of a tropical ideal whose variety is a line, explicitly showing it is not prime. Consider the ideal $J = (x-y) \subseteq k[x,y]$, where $k$ is a trivially valued field. Let $I = trop(J) \subseteq \mb[x,y]$, which is a realizable tropical ideal. Notice that $(x+y+1)(x+y+xy)=(x+y)(x+y+xy+1) \in I$, but neither $(x+y+1)$ nor $(x+y+xy)$ are in $I$, hence $I$ is not prime. %\orange{kalina sanity check $f = 1 + y/x$, let $P_1 = (1, 1)$ and $P_2 = (2, 2)$. Then $f(P_1) = f(P_2) = 0$. $f_1 = f+ t^cx^ay^b$, $f_2 = f+ t^{-c}x^{-a}y^{-b}$. Since we want $f_1(P_2) > f(P_2)$ and $f_2(P_1) > f(P_1)$ this means $c+2a+2b > 0$ and $-c-a-b>0$, so $m = t^{-3}xy$}
\end{example}

%\red{Everything is now only stated for Laurent, the remark below summarizes the polynomial case.}

\begin{remark}\label{rm: noStrNull}
Not every $\closed$ ideal of $\tLaur$ satisfies the strong Nullstellensatz, that is, it is not true that $f$ is in every $\closed$ prime containing a $\closed$ ideal $I$ if and only if some power of $f$ is in $I$.
\end{remark}

\begin{proof}[Proof of Remark~\ref{rm: noStrNull}]
Let $I$ be the closure of the ideal generated by the polynomials
$(1 + x + x^2)y + z, (1 + x + x^2)y + zx, (1 + x + x^2)y + zx^2, (1 + x + x^2)z + y, (1 + x + x^2)z + yx, (1 + x + x^2)z + yx^2$. From the first three polynomials we conclude that 
$$(1 + x + x^2)y \sim (1 + x + x^2)y + (1 + x + x^2)z \in Bend(I).$$
Similarly, from the other three generators we obtain that
$$(1 + x + x^2)z \sim (1 + x + x^2)y + (1 + x + x^2)z \in Bend(I).$$
By transitivity we obtain that the pair $(1 + x + x^2)y \sim (1 + x + x^2)z$ is also in $Bend(I)$ and so $(1 + x + x^2, 0) (y, z)$ in $Bend(I)$. Every prime congruence $P$ over $Bend(I)$ contains $(1 + x + x^2, 0) (y, z)$ hence it contains one of the factors. Assume $(1 + x + x^2,0) \in P$. Since $P$ is a prime congruence on $\tLaur$, in the quotient of $\tLaur$ by $P$ each polynomial is congruent to its leading term; denote the leading term of $1 + x + x^2$ by $m$. So we have $1 + x + x^2 \sim_P m \sim_P 0$. Multiplying by $m^{-1}$ tells us that that $(1_\TT, 0_\TT) \in P$, which makes $P$ not proper and hence not prime. Hence $(y, z)$ is in every prime congruence containing $Bend(I)$. Moreover, $y+z$ is in every $\closed$ prime containing $I$. However, no power of $y + z$ is in $I$. 
\end{proof}

\begin{remark}
   Maclagan and Rinc\'on have announced that tropical ideals satisfy the strong Nullstellensatz; at time of writing, the result is still unpublished.
\end{remark}

\subsection{Dimension}
%\noindent{\color{red} actually prove the below by providing the example of what $I_P$ looks like when P is rank 2 }

In this section we will compute the dimension of the ``coordinate semiring'', i.e., the semiring $\tLaur/Bend(I)$. 

In \cite[Theorem~3.11]{MR16} the authors prove that tropical ideals satisfy the ascending chain condition. Part (iv) of Theorem~\ref{prop: bendPrime-iff-Iprime} implies that in a chain of tropical ideals there would be at most one prime ideal and in a chain of prime ideals there would be at most one tropical ideal. Thus we cannot define Krull dimension in terms of chains of prime tropical ideals. In what follows, Krull dimension is defined in terms of prime congruences as stated in Definition~\ref{def:Krull}.\\

By the Structure Theorem for tropical varieties \cite[Theorem 3.3.5]{MS} the set-theoretic tropicalization $trop(X)$ of a $d$-dimensional subvariety $X$ of $(K^*)^n$ or $K^n$ is a polyhedral complex of pure dimension $d$. Here we relate the dimension of $X$, or rather $trop(X)$, to the Krull dimension of the coordinate ring of the corresponding affine tropical scheme.%\footnote{A version of this result stated for realizable tropical ideals first appeared in the second author's thesis \cite{M16}.}

\begin{theorem}\label{thm: dim} Let $I$ be a tropical ideal in $\tPoly$
%, not necessarily realizable, 
whose variety $X = \VV(I)$ has dimension $d$. Then 
$$\dim \tPoly/Bend(I) = d + 1.$$
\end{theorem}

We will give an upper and lower bound for the dimension of this quotient which will imply the statement of the theorem. We first remind the reader the following terminology. 

\begin{definition}
    An $\mathbb{R}$-rational polyhedron is $P = \{x\in \mathbb{R}^n : Ax \leq b\}$, where $A$ is a $d\times n$ matrix with entries in $\mathbb{Q}$ and $b\in \mathbb{R}^d$.
\end{definition}

\begin{remark}\label{rmk:trop-ideals-facts}
In a series of papers \cite{MR14}, \cite{MR16}, \cite{MR20} Maclagan and Rinc\'on prove that the varieties of tropical ideals are finite balanced $\mathbb{R}$-rational polyhedral complexes. Moreover, in \cite[Theorem 1.4]{MR16} they show that the only tropical ideal whose variety is empty is the polynomial semiring $\tPoly$. Thus every non-trivial tropical ideal has a non-empty variety.
\end{remark}

\begin{proposition}\label{prop: lower-bound}
Let $I$ be an ideal in $\tPoly$ such that $X = \VV(I)$ is a finite dimensional $\mathbb{R}$-rational polyhedral complex of dimension $d$. Then 
$$\dim \tPoly/Bend(I) \geq  d + 1.$$
\end{proposition}

\begin{proof}
Recall that if $P$ is a prime congruence with defining matrix $U$ then $\dim \tPoly/P $ is equal to the rank of $U$, denoted $r(U)$. We first observe that if $C$ is a congruence on $\tPoly$ then $$\dim \tPoly/C = \dim \tPoly/P,$$ where $P$ is a prime congruence over $C$ of maximal rank, in particular, $P$ is a minimal prime over $C$. Thus, if $P$ is a prime of maximal rank over $Bend(I)$ it suffices to show that $P$ has rank at least $d+1$. \par

We will show that there always exists a prime congruence $P$ of rank $d+1$ %with defining matrix $U$ 
containing $Bend(I)$, such that $P$ has a geometric prime lying over it. Let $\mathcal{F}$ be a maximal cell (of maximal dimension) of the polyhedral complex $X$ and let $\omega$ be a point in the relative interior of $\mathcal{F}$. The affine span of $\mathcal{F}$ has dimension equal to the dimension of $X$ which is $d$. Since $X$ is $\mathbb{R}$-rational there exist  $u_1, \dots, u_d  \in \mathcal{F}$ such that $\omega, u_1, \dots, u_d$ are affine independent and $u_1 - \omega, \dots, u_d - \omega$ are rational.
Now consider the matrix $$U = \begin{pmatrix}
    1 & \omega \\
    
    1 & {u_1} \\
    \vdots & \vdots\\
    1 & {u_d}
\end{pmatrix}.$$
Since $\omega, u_1, \dots, u_d$ are affine independent then the rows of $U$ are linearly independent and thus $r(U) = d+1$. Furthermore, $U$ defines the same order as 
$$U' = \begin{pmatrix}
    1 & \omega \\
    0 & {u_1 - \omega} \\
    \vdots & \vdots\\
    0 & {u_d - \omega}
\end{pmatrix}.$$
Since the  $u_1 - \omega, \dots, u_d - \omega$ are rational and linearly independent, each row of $U'$ further refines the ordering hence $U'$, and consequently $U$, defines a prime of rank $d + 1$.
To see that $P$ contains $Bend(I)$, note that for every $f\in I$ and $v\in \VV(I)$ we have that $f(v)$ attains its maximum twice. Since $U$ orders monomials by evaluating them at the points $\omega, u_1, \dots, u_d \in \VV(I)$ it follows that
    $$bend(f) = \{(f , f_{\hat \imath})\}_{i\in supp(f)} \subseteq P,$$
for all $f \in I$, hence $Bend(I) \subseteq P.$

\end{proof}

% \begin{remark}
%     \orange{The above statement and proof hold if the polyhedral complex contains a $\RR$-rational polyhedron of dimension $d$ contained in the torus. The proof can be augmented to drop the torus condition.}
% \end{remark}

\begin{proposition}\label{prop: upper-bound}
Let $I$ be a tropical ideal in $\tPoly$ such that $X = \VV(I)$ has dimension $d$. Then 
$$\dim \tPoly/Bend(I) \leq  d + 1.$$
\end{proposition}
\begin{proof}
We will show that every prime congruence over $Bend(I)$ has rank at most $d+1$. We first see this in the case when every maximal rank prime $P$ has a geometric prime over it. 
%Now assume for contradiction that $P$ is a prime over $Bend(I)$ of rank strictly greater than $d+1$. If $W$ is the defining matrix of $P$ then $r(W) > d+1$. 
Let $W$ be the defining matrix of $P$. Denote the rows of $W$ by $w_1, \dots , w_{r(W)}$. Note that they are linearly independent by definition of the defining matrix of a prime. Consider the vectors $w_1' = w_1,\ w_2' = w_1 + \epsilon_1'w_2,\ \dots ,\ w_{r(W)}' = w_1 + \dots + \epsilon_{r(W)}'w_{r(W)}$ which are also linearly independent. We can scale each of the vectors $w_1', \dots , w_{r(W)}'$ so that the first entry is $1$. Now consider the re-scaled vectors without the first entry and call them $w_1'', \dots, w_{r(W)}''$. Note that the vectors $w_1'', \dots, w_{r(W)}''$ are affine independent and lie on the same face of $X$. 
Since we know that the dimension of $X$ is at most $d$, we know that $r(W) \leq d+1$.
Remains to investigate the case when a maximal rank prime $P$ containing $Bend(I)$ has no geometric prime over it. Denote by $U$ the defining matrix of $P$. If there is no geometric prime over $P$ then the first entry of the first row of $U$ is zero. If the first entry of any other row is not zero, we can add a suitable multiple of this row to the first one. This way we obtain the matrix of a different prime of the same rank, which still contains $Bend(I)$ but has a geometric prime over it. In this case we are done by the previous discussion. %to self: in order for this prime $P$ to contain Bend(I) we need each row to give us the same initial ideal. 

Thus we can assume that the first column of $U$ is the zero vector. Consider the prime $P'$ with matrix $U'$, where
$$U' = \begin{pmatrix}
    1 & {\pmb 0} \\
    {\pmb 0} & U 
\end{pmatrix}.$$
Then there are two cases. First, if $I \subseteq \mb[x_1, \dots, x_n]$, then the prime $P'$ lies above $Bend(I)$ and clearly $r(U')>r(U)$. So $P$ is not a maximal rank prime over $Bend(I)$. %to self: coefficients look like t^0, thus the weight of each is 0 and so is the weight of each monomial. The ordering that P and P' give on B[x..] is the same, but on T[x..] is different. 
If $I \not\subseteq \mb[x_1, \dots, x_n]$, consider the prime $P'$ as before. Since the (1,1) entry of $U'$ is 1, there exists a geometric prime over $P'$. 
% to self: the idea before was that whether we tropicalize wrt the trivial or non-trivial valuation the dimension of the polyhedral complex will stay the same.
We can infer from \cite[Proposition~2.12]{MR20}, that $\dim \VV(I) = \dim \VV(\phi(I)) = d+1$ where the map $\phi: \tPoly \rightarrow \mb[x_1,\dots,x_n]$. Notice that $P'$ is a prime over $Bend(\phi(I))$. Then by the earlier argument $\dim \tPoly/P'$ is at most $d + 1$. Now since $r(U')= r(U) + 1 >r(U)$, then $r(U) < d+1$ and hence $P$ is not a maximal rank prime. So we conclude that if $P$ is a maximal rank prime with matrix $U$, then $r(U) \leq d+1$.
\end{proof}

\begin{proof}[Proof of Theorem~\ref{thm: dim}]
The theorem follows immediately from Proposition~\ref{prop: lower-bound},  Remark~\ref{rmk:trop-ideals-facts}, and Proposition~\ref{prop: upper-bound}. 
\end{proof}

\begin{remark}{\rm
Recently, in \cite{MR20} the authors show that if $I$ is a non-necessarily realizable tropical ideal the dimension of $V(I)$ is equal to the degree of the Hilbert polynomial of $I$. Their result agrees with our computation of the Krull dimension of  $\rma[x_1, \dots, x_n]/Bend(I)$. }
\end{remark}

\section{Appendix}
We give an alternative proof to the result that the only prime tropical ideals are the tropicalizations of points.\par

The following proposition is essentially Theorem 1.1 in \cite{MR14}, which is originally stated for realizable tropical ideals, but is in fact true for all tropical ideals. Here we recall the statement and the proof for the reader's convenience.
\begin{proposition}\label{prop: ideal-bend}
Let $I$ be a homogeneous tropical ideal in $\tPoly$ then $f \in I$ if and only if $bend(f) \subseteq Bend(I)$.
\end{proposition}

\begin{proof}
One direction is easy since $Bend(I)=\left< bend(f) : f\in I\right>.$ For the other, let $f \in I$ be a homogeneous polynomial, then $f$ corresponds to a linear form $l_f$ on the linear space of monomials of degree $d$. Denote by $L_d$ the tropical linear space on which $l_f$ tropically vanishes for all $f \in I_d$ and observe that $L_d = I_d^{\perp}$. Moreover, a polynomial $g$ tropically vanishes on $L_d$ if and only if $l_g(z) = l_{g_{\hat{\imath}}}(z)$, for all $i$ and $z \in L_d$. So if $g$ is such that $bend(g) \subseteq Bend(I)$ then $g \in L_d^{\perp} = I_d$.
\end{proof}

\begin{theorem}\label{thm: mainA}
If $I$ is a non-zero prime tropical ideal of $\tLaur$, then $\VV(I)$ is at most a point. 
\end{theorem}

\begin{proof}
First we show that if $I$ is a tropical prime ideal, then $I = \II(\VV(I))$. By definition it is true that $\II(\VV(I)) \supseteq I$. For the other inclusion from the definition of the sets if $f \in \II(\VV(I))$ then $bend(f) \subseteq Rad (Bend (I))$. However, since $I$ is prime then so is $Bend (I)$ by Proposition~\ref{prop: Iprime-bendPrime}, so $Rad (Bend (I)) = Bend(I)$ implying that $f \in I$.

Now let $\VV(I)$ contain at least two distinct points, call these $P_1$ and $P_2$. We will show that $I$ is not prime. Consider $f \in I$, such that $f \neq 0_{\mathbb{T}}$ and so $\VV(f)$ contains $P_1$ and $P_2$. Such a polynomial exists since by assumption $I$ is non-zero and $|\VV(I)| \geq 2$. Moreover, without loss of generality we can assume that $f$ has a constant term greater than $1_{\mathbb{T}}$. %$f > 1$.

Take $f_1 = f + m$ and $f_2 = f + m^{-1}$, where $m$ is a term. We will show that there exists a choice for $m$ such that $P_i \not\in \VV(f_j), \text{ whenever } i\neq j $ for $i, j = 1, 2$, i.e., that $f_1, f_2 \not\in I$. We will also show that $f_1 f_2\in I$.

First, note that we can choose the term $m = t^cx_1^{u_1}x_n^{u_n}$ so that 
\begin{align*}
m^{-1}(P_1) &> f(P_1) , \text{ and } \\
m(P_2) &> f(P_2),
\end{align*}
because the value of the affine function $m(P) = c + (u_1, \dots, u_n).P$ can be freely set on $n+1$ points in dimension $n$. From our choice of $m$ we also have that
\begin{equation*}
\begin{split}
f_1(P_2) &= m (P_2) > f(P_2)\\
f_2(P_1) &= m^{-1} (P_1) > f(P_1),
\end{split}
\end{equation*}
implying that $P_2 \not\in \VV(f_1)$ and $P_1 \not\in \VV(f_2)$ and so $f_1$ and $f_2$ do not belong to $I$. 
Remains to see that $f_1 f_2\in I$. Indeed,
\begin{equation*}
\begin{split}
    f_1 f_2 &= (f + m)(f + m^{-1}) = f^2 + fm + f(m^{-1}) + 1 \\
            &= f^2 + fm + f(m^{-1}) = f (f + m + (m^{-1})) \in I,
\end{split}
\end{equation*}
where the third equality holds because addition is idempotent and $f$ has a constant term greater than $1_{\mathbb{T}}$.
\end{proof}

% \noindent The following example illustrates the above result.

% \begin{example} We give a simple example of a tropical ideal whose variety is a line, explicitly showing it is not prime. Consider the ideal $J = (x-y) \subseteq k[x,y]$, where $k$ is a trivially valued field. Let $I = trop(J) \subseteq \mb[x,y]$, which is a realizable tropical ideal. Notice that $(x+y+1)(x+y+xy)=(x+y)(x+y+xy+1) \in I$, but neither $(x+y+1)$ nor $(x+y+xy)$ are in $I$, hence $I$ is not prime. %\orange{kalina sanity check $f = 1 + y/x$, let $P_1 = (1, 1)$ and $P_2 = (2, 2)$. Then $f(P_1) = f(P_2) = 0$. $f_1 = f+ t^cx^ay^b$, $f_2 = f+ t^{-c}x^{-a}y^{-b}$. Since we want $f_1(P_2) > f(P_2)$ and $f_2(P_1) > f(P_1)$ this means $c+2a+2b > 0$ and $-c-a-b>0$, so $m = t^{-3}xy$}
% \end{example}

%\newpage
\bibliographystyle{alpha}

\end{document}